\documentclass[11pt,leqno]{article}
\usepackage[active]{srcltx}

\usepackage{amsfonts,latexsym,amsmath,amssymb,amsthm}
\usepackage{fullpage}
\newtheorem{theorem}{Theorem}[section]
\newtheorem{lemma}[theorem]{Lemma}

\newtheorem{proposition}[theorem]{Proposition}

\newtheorem{remark}[theorem]{Remark}

\theoremstyle{definition}

\theoremstyle{remark}

\newtheorem*{note*}{Note}

\numberwithin{equation}{section}

\makeatletter
\newcommand{\rank}{\mathop{\operator@font rank}}
\newcommand{\conv}{\mathop{\operator@font conv}}
\newcommand{\vol}{\mathop{\operator@font vol}}
\newcommand{\onetagright}{\tagsleft@false}
\makeatother

\newcommand{\ls}{\leqslant}
\newcommand{\gr}{\geqslant}
\renewcommand{\epsilon}{\varepsilon}

\usepackage{color}

\usepackage{hyperref}

\begin{document}
\small

\title{\bf Fractional Helly theorem for the diameter of convex sets}

\medskip

\author{Silouanos Brazitikos}

\date{}

\maketitle

\begin{abstract}
\footnotesize We provide a new quantitative version of Helly's theorem: there exists an absolute constant $\alpha >1$ with the following
property: if $\{P_i: i\in I\}$ is a finite family of convex bodies in ${\mathbb R}^n$ with ${\rm int}\left (\bigcap_{i\in I}P_i\right )\neq\emptyset $,
then there exist $z\in {\mathbb R}^n$, $s\ls \alpha n$ and $i_1,\ldots i_s\in I$ such that
\begin{equation*}
z+P_{i_1}\cap\cdots\cap P_{i_s}\subseteq cn^{3/2}\left(z+\bigcap_{i\in I}P_i\right),
\end{equation*}
where $c>0$ is an absolute constant. This directly gives a version of the ``quantitative" diameter theorem of
B\'{a}r\'{a}ny, Katchalski and Pach, with a polynomial dependence on the dimension. In the symmetric case the bound $O(n^{3/2})$ can be improved
to $O(\sqrt{n})$.
\end{abstract}

\section{Introduction}

The purpose of this work is to present a new quantitative versions of Helly's theorem; recall that the classical
result asserts that if ${\mathcal F}=\{ F_i:i\in I\}$ is a finite family of at least $n+1$ convex sets in ${\mathbb R}^n$
and if any $n+1$ members of ${\mathcal F}$ have non-empty intersection then $\bigcap_{i\in I}F_i\neq\emptyset $. Variants of
this statement have found important applications in discrete and computational geometry.
Quantitative Helly-type results were first obtained by B\'{a}r\'{a}ny, Katchalski and Pach in \cite{BKP-1982}
(see also \cite{BKP-1984}). In particular, they proved the following volumetric result:

\smallskip

{\sl Let $\{ P_i:i\in I\}$ be a family of closed convex sets in ${\mathbb R}^n$
such that $\left |\bigcap_{i\in I}P_i\right |>0$. There exist $s\ls 2n$ and $i_1,\ldots ,i_s\in I$ such that
\begin{equation}\left |P_{i_1}\cap \cdots \cap P_{i_s}\right |\ls C_n\left |\bigcap_{i\in I} P_i\right |,\end{equation}
where $C_n>0$ is a constant depending only on $n$.}

\smallskip

The example of the cube $[-1,1]^n$ in ${\mathbb R}^n$, expressed as an intersection of exactly $2n$ closed half-spaces,
shows that one cannot replace $2n$ by $2n-1$ in the statement above. The optimal growth of the constant $C_n$ as a function
of $n$ is not completely understood. The bound in \cite{BKP-1982} was $O(n^{2n^2})$ and it was conjectured that one might actually
have $C_n\ls n^{cn}$ for an absolute constant $c>0$. Nasz\'{o}di \cite{Naszodi-2015} has recently proved
a volume version of Helly's theorem with $C_n\ls (cn)^{2n}$, where $c>0$ is an absolute constant. In fact, a slight modification of
Nasz\'{o}di's argument leads to the exponent $\frac{3n}{2}$ instead of $2n$. In \cite{Brazitikos-2015}, relaxing the requirement that $s\ls 2n$ to the
weaker one that $s=O(n)$, we were able to improve the exponent to $n$:

\begin{theorem}[Brazitikos]\label{th:volume}There exists an absolute constant $\alpha >1$ with the following
property: for every family $\{P_i:i\in I\}$ of closed convex sets
in ${\mathbb R}^n$, such that $P=\bigcap_{i\in I}P_i$ has positive volume, there exist $s\ls \alpha n$ and
$i_1,\ldots , i_s\in I$ such that
\begin{equation}|P_{i_1}\cap\cdots\cap P_{i_s}|\ls (cn)^n\,|P|,\end{equation}
where $c>0$ is an absolute constant.
\end{theorem}

Following the terminology of \cite{LHRS-2015}, a Helly-type property is (loosely speaking) a property $\Pi $ for which there exists $m\in {\mathbb N}$ such that
if $\{ P_i:i\in I\}$ is a finite family of objects such that every subfamily with $m$ elements satisfies $\Pi $, then the whole family satisfies $\Pi $. Thus, the
previous results (in particular, Theorem \ref{th:volume}) express the fact that the property that ``an intersection has large volume" is a Helly-type
property for the class of convex sets.

B\'{a}r\'{a}ny, Katchalski and Pach studied the question if the property that ``an intersection has large diameter" is also a Helly-type
property for the class of convex sets. In \cite{BKP-1982} they gave a first quantitative answer to this question:

\smallskip

{\sl Let $\{ P_i:i\in I\}$ be a family of closed convex sets in ${\mathbb R}^n$
such that ${\rm diam}\left (\bigcap_{i\in I}P_i\right )=1$. There exist $s\ls 2n$ and $i_1,\ldots ,i_s\in I$ such that
\begin{equation}{\rm diam}\left (P_{i_1}\cap \cdots \cap P_{i_s}\right )\ls (cn)^{n/2},\end{equation}
where $c>0$ is an absolute constant.}

\smallskip

In the same work the authors conjecture that the bound should be polynomial in $n$; in fact they ask if $(cn)^{n/2}$ can be replaced
by $c\sqrt{n}$. Relaxing the requirement that $s\ls 2n$, exactly as in \cite{Brazitikos-2015}, we provide a positive answer, although we are not able to
achieve a bound of the order of $\sqrt{n}$.

Starting with the symmetric case, our main result is the next theorem.

\begin{theorem}\label{th:main-sym}
Let $\{P_i: i\in I\}$ be a finite family of symmetric convex sets in ${\mathbb R}^n$ with ${\rm int}\left (\bigcap_{i\in I}P_i\right )\neq\emptyset $.
For every $d>1$ there exist $s\ls dn$ and $i_1,\ldots i_s\in I$ such that
\begin{equation}\label{eq:main-sym-1}
P_{i_1}\cap\cdots\cap P_{i_s}\subseteq \gamma_d\sqrt{n}\left(\bigcap_{i\in I}P_i\right),
\end{equation}
where $\gamma_d:=\frac{\sqrt{d}+1}{\sqrt{d}-1}$.
\end{theorem}

The proof of Theorem \ref{th:main-sym} is presented in Section 3; it is based on a lemma of Barvinok from \cite{Barvinok-2014} which, in turn, exploits a theorem of Batson, Spielman and Srivastava from \cite{BSS-2009}. It is clear that the $\sqrt{n}$-dependence cannot be improved (we provide a simple example).

In the general (not necessarily symmetric) case, using a similar strategy and ideas that were developed in \cite{Brazitikos-2015} and employ
a more delicate theorem of Srivastava from \cite{Srivastava-2012}, we obtain the next estimate.

\begin{theorem}\label{th:main-gen}
There exists an absolute constant $\alpha >1$ with the following
property: if $\{P_i: i\in I\}$ is a finite family of convex bodies in ${\mathbb R}^n$ with ${\rm int}\left (\bigcap_{i\in I}P_i\right )\neq\emptyset $,
then there exist $z\in {\mathbb R}^n$, $s\ls \alpha n$ and $i_1,\ldots i_s\in I$ such that
\begin{equation}\label{eq:main-gen-intro}
z+P_{i_1}\cap\cdots\cap P_{i_s}\subseteq cn^{3/2}\left(z+\bigcap_{i\in I}P_i\right),
\end{equation}
where $c>0$ is an absolute constant.
\end{theorem}

It is clear that Theorem \ref{th:main-sym} and Theorem \ref{th:main-gen} imply polynomial estimates for the diameter:

\begin{theorem}\label{th:diameter}{\rm (a)} Let $\{P_i: i\in I\}$ be a finite family of symmetric convex sets in ${\mathbb R}^n$ with ${\rm diam}\left (\bigcap_{i\in I}P_i\right )=1$.
For every $d>1$ there exist $s\ls dn$ and $i_1,\ldots i_s\in I$ such that
\begin{equation}\label{eq:diameter-1}
{\rm diam}(P_{i_1}\cap\cdots\cap P_{i_s})\ls \gamma_d\sqrt{n},
\end{equation}
where $\gamma_d:=\frac{\sqrt{d}+1}{\sqrt{d}-1}$.

\smallskip

\noindent {\rm (b)} There exists an absolute constant $\alpha >1$ with the following
property: if $\{P_i: i\in I\}$ is a finite family of convex bodies in ${\mathbb R}^n$ with ${\rm diam}\left (\bigcap_{i\in I}P_i\right )=1$,
then there exist $s\ls \alpha n$ and $i_1,\ldots i_s\in I$ such that
\begin{equation}\label{eq:diameter-2}
{\rm diam}(P_{i_1}\cap\cdots\cap P_{i_s})\ls cn^{3/2},
\end{equation}
where $c>0$ is an absolute constant.
\end{theorem}

\section{Notation and background}

We work in ${\mathbb R}^n$, which is equipped with a Euclidean structure $\langle\cdot ,\cdot\rangle $. We denote by $\|\cdot \|_2$
the corresponding Euclidean norm, and write $B_2^n$ for the Euclidean unit ball and $S^{n-1}$ for the unit sphere.
Volume is denoted by $|\cdot |$. We write $\omega_n$ for the volume of $B_2^n$ and $\sigma $ for the rotationally invariant probability
measure on $S^{n-1}$. We will denote by $P_F$ the orthogonal projection from $\mathbb R^{n}$ onto $F$. We also define
$B_F=B_2^n\cap F$ and $S_F=S^{n-1}\cap F$.

The letters $c,c^{\prime }, c_1, c_2$ etc. denote absolute positive constants which may change from line to line. Whenever we write
$a\simeq b$, we mean that there exist absolute constants $c_1,c_2>0$ such that $c_1a\ls b\ls c_2a$.  Also, if $K,L\subseteq \mathbb R^n$
we will write $K\simeq L$ if there exist absolute constants $c_1,c_2>0$ such that $ c_{1}K\subseteq L \subseteq c_{2}K$.

We refer to the book of Schneider \cite{Schneider-book} for basic facts from the Brunn-Minkowski theory and to the book
of Artstein-Avidan, Giannopoulos and V. Milman \cite{AGA-book} for basic facts from asymptotic convex geometry.

\smallskip

A convex body in ${\mathbb R}^n$ is a compact convex subset $K$ of ${\mathbb R}^n$ with non-empty interior. We say that $K$ is
symmetric if $x\in K$ implies that $-x\in K$, and that $K$ is centered if its barycenter
\begin{equation}{\rm bar}(K)=\frac{1}{|K|}\int_Kx\,dx \end{equation} is at the origin. The circumradius of $K$ is the radius of the
smallest ball which is centered at the origin and contains $K$:
\begin{equation}R(K)=\max\{ \|x\|_2:x\in K\}.\end{equation}
If $0\in {\rm int}(K)$ then the polar body $K^{\circ }$ of $K$ is defined by
\begin{equation}K^{\circ }:=\{ y\in {\mathbb R}^n: \langle x,y\rangle \ls 1
\;\hbox{for all}\; x\in K\}. \end{equation}
and the Minkowski functional of $K$ is defined by
\begin{equation}p_K(x)=\min\{ t\gr 0: x\in tK\}.\end{equation}
Recall that $p_K$ is subadditive and positively homogeneous.

\smallskip

We say that a convex body $K$ is in John's position if the ellipsoid of maximal volume inscribed in $K$ is
the Euclidean unit ball $B_2^n$. John's theorem \cite{John-1948} states that
$K$ is in John's position if and only if $B_2^n\subseteq K$ and there exist $v_1,\ldots ,v_m\in {\rm
bd}(K)\cap S^{n-1}$ (contact points of $K$ and $B_2^n$) and positive real numbers $a_1,\ldots ,a_m$ such that
\begin{equation}\label{eq:bar-0}\sum_{j=1}^ma_jv_j=0\end{equation}
and the identity operator $I_n$ is decomposed in the form
\begin{equation}\label{eq:decomposition}I_n=\sum_{j=1}^ma_jv_j\otimes v_j, \end{equation}
where $(v_j\otimes v_j)(y)=\langle v_j,y\rangle v_j$. In the case where $K$ is symmetric, the second
condition \eqref{eq:decomposition} is enough (for any contact point $u$ we have that $-u$ is also
a contact point, and hence, having \eqref{eq:decomposition} we may easily produce a decomposition for which
\eqref{eq:bar-0} is also satisfied). In analogy to John's position, we say that a convex body $K$ is in L\"{o}wner's
position if the ellipsoid of minimal volume containing $K$ is the Euclidean unit ball
$B_2^n$. One can check that this holds true if and only if $K^{\circ }$ is in John position;
 in particular, we have a decomposition of the identity similar to
\eqref{eq:decomposition}.

Assume that $v_1,\ldots ,v_m$ are unit vectors that satisfy John's decomposition \eqref{eq:decomposition} with some
positive weights $a_j$. Then, one has the useful identities
\begin{equation}\label{eq:trace}\sum_{j=1}^ma_j={\rm tr}(I_n)=n\quad \hbox{and}\quad
\sum_{j=1}^ma_j\langle v_j,z\rangle^2=1 \end{equation}
for all $z\in S^{n-1}$. Moreover,
\begin{equation}\label{eq:ball-inside}{\rm conv}\{v_1,\ldots ,v_m\}\supseteq\frac{1}{n}B_2^n.\end{equation}
In the symmetric case we actually have
\begin{equation}\label{eq:ball-inside-sym}{\rm conv}\{\pm v_1,\ldots ,\pm v_m\}\supseteq\frac{1}{\sqrt{n}}B_2^n.\end{equation}

\section{Symmetric case}

Our main tool for the symmetric case is a lemma of Barvinok from \cite{Barvinok-2014}, which exploits the next
theorem of Batson, Spielman and Srivastava \cite{BSS-2009} on extracting an approximate John's decomposition with
few vectors from a John's decomposition of the identity.

\begin{theorem}[Batson-Spielman-Srivastava]\label{th:BSS}Let $v_1,\ldots ,v_m \in S^{n-1}$ and $a_1,\ldots ,a_m>0$ such that
\begin{equation}\label{eq:BSS-1}I_n=\sum_{j=1}^{m}a_jv_j\otimes v_j.\end{equation}
Then, for every $d>1$ there exists a subset $\sigma\subseteq \{ 1,\ldots ,m\}$ with $|\sigma |\ls dn$ and $b_j>0$, $j\in \sigma $, such that
\begin{equation}\label{eq:BSS-2}I_n\preceq \sum_{j\in\sigma }b_ja_jv_j\otimes v_j \preceq \gamma_d^2I_n,\end{equation}
where $\gamma_d:=\frac{\sqrt{d}+1}{\sqrt{d}-1}$.
\end{theorem}

Here, given two symmetric positive definite matrices $A$ and $B$ we write $A\preceq B$ if $\langle Ax,x\rangle \ls \langle Bx,x\rangle $
for all $x\in {\mathbb R}^n$. Barvinok's lemma is the next statement.

\begin{lemma}[Barvinok]\label{lem:barvinok}
Let $C\subset\mathbb{R}^n$ be a compact set. Then, there exists a subset $X\subseteq C$ of
cardinality ${\rm card}(X)\ls dn$ such that for any $z\in {\mathbb R}^n$ we have
\begin{equation}\label{eq:barvinok}
\max_{x\in X}|\langle z,x\rangle |\ls\max_{x\in C}|\langle z,x\rangle|\ls \gamma_d\sqrt{n}\max_{x\in X}|\langle z,x\rangle|
\end{equation}
\end{lemma}

\begin{proof}We sketch the proof for completeness. We may assume that $C$ spans ${\mathbb R}^n$ and, by the linear invariance of
the statement, that $B_2^n$ is the origin symmetric ellipsoid of minimal volume containing $C$.
Then, there exist $v_1,\ldots ,v_m\in C\cap S^{n-1}$ and $a_1,\ldots ,a_m>0$ such that \eqref{eq:BSS-1} is satisfied. Then, applying
Theorem \ref{th:BSS} we may find a subset $\sigma\subseteq \{ 1,\ldots ,m\}$ with ${\rm card}(\sigma )\ls dn$ and $b_j>0$, $j\in \sigma $, such that
\eqref{eq:BSS-2} holds true. In particular,
\begin{equation}\label{eq:bar-1}n\ls \sum_{j\in\sigma }a_jb_j={\rm tr}\left (\sum_{j\in\sigma }b_ja_jv_j\otimes v_j\right )\ls\gamma_d^2n.\end{equation}
Given $z\in {\mathbb R}^n$, from \eqref{eq:BSS-2} and \eqref{eq:bar-1} we have
\begin{equation}\label{eq:bar-2}\|z\|_2^2\ls \sum_{j\in\sigma }b_ja_j\langle z,v_j\rangle^2\ls \gamma_d^2n\,\max_{j\in \sigma }|\langle z,v_j
\rangle |^2,\end{equation}
and using the fact that $C\subseteq B_2^n$ we conclude that
\begin{equation}\label{eq:bar-3}\max_{x\in C}|\langle z,x\rangle|\ls\|z\|\ls\gamma_d\sqrt{n}\,\max_{j\in\sigma }|\langle z,v_j\rangle |.\end{equation}
Setting $X=\{ v_j:j\in\sigma \}$ we conclude the proof.
\end{proof}

\medskip

Using Barvinok's lemma we can prove Theorem \ref{th:main-sym}.

\medskip

\noindent {\bf Proof of Theorem \ref{th:main-sym}.} Let $P=\bigcap_{i\in I}P_i$ and consider its polar body
\begin{equation}\label{eq:main-sym-2}P^{\circ}=\mathrm{conv}\left(\bigcup_{i\in I} P_i^{\circ}\right).\end{equation}
Using Lemma \ref{lem:barvinok} for $C=P^{\circ }$ we may find $X=\{ v_1,\ldots ,v_s\}\subset P^{\circ }$ with ${\rm card}(X)=s\ls dn$ such that
\begin{equation}\label{eq:main-sym-3}\max_{x\in P^{\circ }}|\langle z,x\rangle|\ls \gamma_d\sqrt{n}\,\max_{x\in X}|\langle z,x\rangle|\end{equation}
for all $z\in {\mathbb R}^n$. It follows that
\begin{equation}\label{eq:main-sym-4}P^{\circ}\subseteq\gamma_d\sqrt{n}\,\mathrm{conv}(\{\pm v_1,\ldots ,\pm v_s\}).\end{equation}
From the proof of Lemma \ref{lem:barvinok} we see that $v_1,\ldots ,v_s$ may be chosen to be
contact points of $P^{\circ }$ with its minimal volume ellipsoid, and hence it is simple to check that we actually have
$v_j\in \bigcup_{i\in I}P_i^{\circ }$ for all $j=1,\ldots ,s$. In other words, we may find $i_1,\ldots ,i_s\in I$ such
that $v_j\in P_{i_j}$, $j=1,\ldots ,s$. Then, \eqref{eq:main-sym-4} implies that
\begin{equation}\label{eq:main-sym-5}P^{\circ}\subseteq\gamma_d\sqrt{n}\,\mathrm{conv}(P_{i_1}^{\circ }\cup\cdots \cup P_{i_s}^{\circ }),\end{equation}
and passing to the polar bodies, we get
\begin{equation}\label{eq:main-sym-6}P_{i_1}\cap\cdots\cap P_{i_s}\subseteq\gamma_d\sqrt{n}\,P\end{equation}
as claimed. $\hfill\Box $

\begin{remark}\rm Theorem \ref{th:main-sym} is sharp in the following sense: we can find $w_1,\ldots ,w_N\in S^{n-1}$ (assuming that $N$ is exponential in the dimension $n$) such that
\begin{equation}B_2^n\subseteq \bigcap_{j=1}^NP_j\subseteq 2B_2^n,\end{equation}
where
\begin{equation}P_j=\{ x\in {\mathbb R}^n:|\langle x,w_j\rangle |\ls 1\}.\end{equation}
For any $s\ls dn$ and any choice of $j_1,\ldots ,j_s\in \{ 1,\ldots ,N\}$,
well-known lower bounds for the volume of intersections of strips, due
to Carl-Pajor \cite{Carl-Pajor-1988}, Gluskin \cite{Gluskin-1988} and Ball-Pajor \cite{Ball-Pajor-1990}
show that
\begin{equation}\label{lower}|P_{j_1}\cap\cdots \cap P_{j_s}|^{1/n}\gr \frac{2}{\sqrt{e}\sqrt{\log (1+d)}}.\end{equation}
Therefore, if $P_{j_1}\cap \cdots \cap P_{j_s}\subseteq \alpha\bigcap_{j=1}^NP_j$ for some $\alpha >0$, comparing volumes we see that
\begin{equation}\alpha \gr \frac{|P_{j_1}\cap\cdots \cap P_{j_s}|^{1/n}}{|2B_2^n|^{1/n}}\gr \frac{c}{\sqrt{\log (1+d)}}\sqrt{n},\end{equation}
where $c>0$ is an absolute constant.
\end{remark}

\section{General case}

In order to deal with the not-necessarily symmetric case we use the next theorem of Srivastava from \cite{Srivastava-2012}.

\begin{theorem}[Srivastava]\label{th:sriv}Let $v_1,\ldots ,v_m\in S^{n-1}$ and $a_1,\ldots ,a_m>0$ such that
\begin{equation}I_n=\sum_{j=1}^ma_jv_j\otimes v_j\quad\hbox{and}\quad \sum_{j=1}^ma_jv_j=0.\end{equation}
Given $\varepsilon>0$ we can find a subset $\sigma $ of $\{1,\ldots ,m\}$ of cardinality $|\sigma |=O_{\varepsilon}(n)$,
positive scalars $b_i$, $i\in \sigma $ and a vector $v$ with
\begin{equation}\label{eq:vector-v}\|v\|_2^2\ls\frac{\varepsilon}{\sum_{i\in \sigma } b_i},\end{equation}
such that
\begin{equation}\label{Srivastavalemma2}
I_n\preceq\sum_{i\in \sigma } b_i(v_i+v)\otimes (v_i+v)\preceq (4+\varepsilon)I_n
\end{equation}
and \begin{equation}\label{Srivastavalemma3}\sum_{i\in \sigma } b_i(v_i+v)=0.\end{equation}
\end{theorem}

\begin{proposition}\label{prop:main-gen}
There exists an absolute constant $\alpha >1$ with the following
property: if $K$ is a convex body whose minimal volume ellipsoid is the Euclidean unit ball,
then there is a subset $X\subset K$ of
cardinality ${\rm card}(X)\ls \alpha n$ such that
\begin{equation}\label{eq:barvinok-gen-1}
B_2^n\subseteq cn^{3/2}{\rm conv}(X),
\end{equation}
where $c>0$ is an absolute constant.
\end{proposition}

\begin{proof}As in the proof of Lemma \ref{lem:barvinok} we assume that $B_2^n$ is the minimal volume ellipsoid
of $K$, and we find $v_j\in K\cap S^{n-1}$ and $a_j>0$, $j\in J$, such that
\begin{equation}\label{eq:barvinok-gen-2}I_n=\sum_{j\in J}a_jv_j\otimes v_j\quad\hbox{and}\quad \sum_{j\in J}a_jv_j=0.\end{equation}
We fix $\varepsilon >0$, which will be chosen small enough, and we apply Theorem \ref{th:sriv} to find a subset
$\sigma \subseteq J$ with $|\sigma |\ls\alpha_1(\varepsilon )n$, positive scalars $b_j$, $j\in \sigma $ and a vector $v$ such that
\begin{equation}\label{half11}
I_n\preceq\sum_{j\in \sigma} b_j(v_j+v)\otimes (v_j+v)\preceq (4+\varepsilon )I_n
\end{equation}
and \begin{equation}\label{half12}\sum_{j\in \sigma } b_j(v_j+v)=0\ \ \text{and}\ \ \|v\|_2^2\ls\frac{\varepsilon }{\sum_{j\in \sigma } b_j}.\end{equation}
Note that
\begin{equation}\label{eq:barvinok-gen-3}
{\rm tr}\left(\sum_{j\in \sigma} b_j(v_j+v)\otimes (v_j+v)\right)=\sum_{j\in \sigma }b_j-\left (\sum_{j\in\sigma }b_j\right )\|v\|_2^2
\end{equation}\label{eq:barvinok-gen-4}
and hence from \eqref{half11} we get that
$$n\ls\sum_{j\in \sigma }b_j-\left (\sum_{j\in\sigma }b_j\right )\|v\|_2^2\ls (4+\varepsilon )n.$$
Now, using \eqref{half12} we get
\begin{equation}\label{ineq1}
n\ls\sum_{j\in \sigma}b_j\ls (4+2\varepsilon )n.
\end{equation}
In particular,
\begin{equation}\label{eq:barvinok-gen-5}\|v\|_2^2\ls\frac{\varepsilon }{\sum_{j\in\sigma }b_j}\ls\frac{\varepsilon }{n}.\end{equation}
From John's theorem we know that $\conv\{v_j,j\in J\}\supseteq\frac{1}{n}B_2^n$. Then, for the vector $w=\frac{v}{\sqrt{\varepsilon n}}$ we have $\|w\|_2\ls\frac{1}{n}$ and hence $w\in \conv\{v_j, j\in J\}$. Carath\'{e}odory's theorem shows that there exist $\tau\subseteq J$ with $|\tau |\ls n+1$ and $\rho_i> 0$, $i\in\tau $ such that
\begin{equation}\label{eq:barvinok-gen-6}w=\sum_{i\in\tau }\rho_iv_i\ \ \text{and}\ \ \sum_{i\in\tau }\rho_i=1.\end{equation}
Note that \begin{equation}\label{eq:barvinok-gen-7}\left (\sum_{j\in \sigma}b_j\right )(-v)=\sum_{j\in \sigma }b_jv_j,\end{equation}
and this shows that $-v\in {\rm conv}\{v_j:j\in \sigma\}$.

We write
\begin{equation}\label{eq:6-44}
I_{n}-T\preceq\sum_{j\in \sigma} b_jv_j\otimes v_j\preceq (4+2\varepsilon )I_{n}-T,\end{equation}
where
\begin{equation}\label{eq:barvinok-gen-8}T:=\sum_{j\in \sigma} b_jv_j\otimes v+\sum_{j\in \sigma} v\otimes b_jv_j+\left(\sum_{j\in \sigma} b_j\right)v\otimes v.\end{equation}
Taking into account \eqref{eq:barvinok-gen-7} we check that, for every $x\in S^{n-1}$,
\begin{equation}|\langle Tx,x\rangle |= \left(\sum_{j\in \sigma} b_j\right)\langle x,v\rangle^2\ls \left(\sum_{j\in \sigma} b_j\right)\|v\|_2^2\ls\varepsilon .\end{equation}
Choosing $\varepsilon =1/2$ we see that $\|T\|\ls \frac{1}{2}$, and this finally gives
\begin{equation}\label{eq:barvinok-gen-9}\frac{1}{2}I_n\preceq A:=\sum_{j\in\sigma }b_jv_j\otimes v_j\preceq \frac{11}{2}I_n.\end{equation}
We are now able to show that
\begin{equation}\label{eq:contain}K:={\rm conv}(\{ v_j:j\in\sigma\cup\tau \})\supseteq \frac{c}{n^{3/2}}B_2^n.\end{equation}
Let $x\in S^{n-1}$. We set $\delta =\min\{ \langle x,v_j\rangle :j\in\sigma \}$; note that $|\delta |\ls 1$ and $\langle x,v_j\rangle -\delta\ls 2$
for all $j\in\sigma $. If $\delta <0$, we write
\begin{align*}
p_K(Ax) &\ls p_K\left (Ax-\delta\sum_{j\in\sigma }b_jv_j\right )+p_K\left (\delta \sum_{j\in\sigma }b_jv_j\right )\\
&= p_K\left (\sum_{j\in\sigma }b_j(\langle x,v_j\rangle -\delta )v_j\right )+p_K\left (\delta \Big(\sum_{j\in\sigma }b_j\Big ) (-v)\right )\\
&\ls \sum_{j\in\sigma }b_j(\langle x,v_j\rangle -\delta )p_K(v_j)-\delta \Big(\sum_{j\in\sigma }b_j\Big ) p_K(v)\\
&\ls \Big(\sum_{j\in\sigma }b_j\Big )\left [2+\sqrt{n/2}p_K(w)\right ]\\
&\ls c_1n^{3/2},
\end{align*}
using the fact that $w\in K$, and hence $p_K(w)\ls 1$. If $\delta \gr 0$ then $\langle x,v_j\rangle\gr 0$ for all $j\in\sigma $,
therefore
\begin{equation}p_K(Ax)=p_K\left (\sum_{j\in\sigma }b_j\langle x,v_j\rangle v_j\right )\ls \sum_{j\in\sigma }b_j\langle x,v_j\rangle p_K(v_j)
\ls \sum_{j\in\sigma }b_j\ls 5n\end{equation}
In any case,
\begin{equation}p_{A^{-1}(K)}(x)\ls c_2n^{3/2}\end{equation}
for all $x\in S^{n-1}$, where $c_2>0$ is an absolute constant. Together with \eqref{eq:barvinok-gen-9} this shows that
\begin{equation}\frac{1}{2}B_2^n\subseteq A(B_2^n)\subseteq c_2n^{3/2}K.\end{equation}
Since ${\rm card}(\sigma\cup\tau )\ls \alpha_1(1/2)n+n+1\ls (\alpha_1(1/2)+2)n$, the
proof is complete. \end{proof}

\medskip

\noindent {\bf Proof of Theorem \ref{th:main-gen}.} Let $P=\bigcap_{i\in I}P_i$. We may assume that $0\in {\rm int}(P)$ and that the polar body
\begin{equation}\label{eq:main-gen-1}P^{\circ}=\mathrm{conv}\left(\bigcup_{i\in I} P_i^{\circ}\right)\end{equation}
is in L\"{o}wner's position. Using Proposition \ref{prop:main-gen} for $C=P^{\circ }$ we may find $X=\{ v_1,\ldots ,v_s\}\subset P^{\circ }$ with ${\rm card}(X)=s\ls \alpha n$ such that
\begin{equation}\label{eq:main-gen-2}P^{\circ}\subseteq cn^{3/2}\mathrm{conv}(\{v_1,\ldots ,v_s\}),\end{equation}
where $c>0$ is an absolute constant. From the proof of Proposition \ref{prop:main-gen} we see that $v_1,\ldots ,v_s$ may be chosen to be
contact points of $P^{\circ }$ with its minimal volume ellipsoid, and hence it is simple to check that we actually have
$v_j\in \bigcup_{i\in I}P_i^{\circ }$ for all $j=1,\ldots ,s$. In other words, we may find $i_1,\ldots ,i_s\in I$ such
that $v_j\in P_{i_j}$, $j=1,\ldots ,s$. Then, \eqref{eq:main-gen-2} implies that
\begin{equation}\label{eq:main-gen-3}P^{\circ}\subseteq cn^{3/2}\mathrm{conv}(P_{i_1}^{\circ }\cup\cdots \cup P_{i_s}^{\circ }),\end{equation}
and passing to the polar bodies, we get
\begin{equation}\label{eq:main-gen-4}P_{i_1}\cap\cdots\cap P_{i_s}\subseteq cn^{3/2}P\end{equation}
as claimed. $\hfill\Box $

\begin{remark}\rm In \cite{BKP-1982} it is proved that if $\{P_i: i\in I\}$ is a finite family of convex bodies in ${\mathbb R}^n$ with ${\rm diam}\left (\bigcap_{i\in I}P_i\right )=1$, then there exist $s\ls n(n+1)$ and $i_1,\ldots i_s\in I$ such that
\begin{equation}\label{eq:diameter-rem-1}
{\rm diam}(P_{i_1}\cap\cdots\cap P_{i_s})\ls \sqrt{2n(n+1)}.
\end{equation}
Then, a scheme is described which allows one to further reduce the number of the bodies $P_{i_j}$ and keep some control on the diameter. The lemma
which allows this reduction states the following: Let $m>2n$ and $P_1,\ldots ,P_m$ be convex bodies in ${\mathbb R}^n$ such that $0\in P_1\cap\cdots\cap P_m$.
If the circumradius of $P_1\cap\cdots\cap P_m$ is equal to $1$ then we can find $1\ls j\ls m$ such that the circumradius of $\bigcap_{i=1,i\neq j}^mP_i$ is at most
$\frac{m}{m-2d}$. Starting with Theorem \ref{th:main-gen} and using the same lemma, for any finite family $\{P_i: i\in I\}$ of convex bodies 
in ${\mathbb R}^n$ with ${\rm diam}\left (\bigcap_{i\in I}P_i\right )=1$ we first find $s\ls \alpha n$ and $i_1,\ldots i_s\in I$ such that
\begin{equation}\label{eq:diameter-rem-2}
{\rm diam}(P_{i_1}\cap\cdots\cap P_{i_s})\ls c_1n^{3/2},
\end{equation}
where $c_1>0$ is an absolute constant, and then we can keep $2n$ of the $P_{i_j}$'s so that the diameter of their intersection is bounded by
\begin{equation}c_1n^{3/2}\prod_{m=2n+1}^s\frac{m}{m-2n}=cn^{3/2}\binom{s}{2n}\ls cn^{3/2}\left (\frac{e\alpha }{2}\right )^{2n}\ls c_2^n,
\end{equation}
where $c_2>0$ is an absolute constant. This improves the estimate from \cite{BKP-1982} (for the original question studied there) but it is still
exponential in the dimension.
\end{remark}

\bigskip

\noindent {\bf Acknowledgement.} I acknowledge founding support from Onassis Foundation. I would like to thank Apostolos Giannopoulos for useful discussions.

\bigskip

\bigskip

\footnotesize
\bibliographystyle{amsplain}

\medskip

\thanks{\noindent {\bf Keywords:}  Convex bodies, Helly's theorem, approximate John's decomposition, polytopal approximation.}

\smallskip

\thanks{\noindent {\bf 2010 MSC:} Primary 52A23; Secondary 52A35, 46B06.}

\bigskip

\bigskip

\noindent \textsc{Silouanos \ Brazitikos}: Department of
Mathematics, University of Athens, Panepistimioupolis 157-84,
Athens, Greece.

\smallskip

\noindent \textit{E-mail:} \texttt{silouanb@math.uoa.gr}

\end{document}